\documentclass[a4paper]{article}

\usepackage{zibtitlepage}
\usepackage{graphicx}
\usepackage{amsfonts,amssymb,amsmath,textcomp,mathabx,empheq,amsthm}
\usepackage{bm}
\usepackage{color}
\usepackage[bookmarks=false]{hyperref}
\usepackage{url}

\usepackage{todonotes}

\usepackage{booktabs}
\usepackage{array}
\usepackage{doi}
\usepackage{subcaption}
\usepackage{lmodern}
\usepackage{authblk}

\newtheorem{definition}{Definition}
\newtheorem{theorem}{Theorem}
\newtheorem{example}{Example}

\hypersetup{
	colorlinks=false,
	breaklinks=true,
	urlcolor= blue, 
	linkcolor= blue, 
	citecolor=red, 
	pdftitle={},
	pdfauthor={},
}

\def\rit{\mathbb{R}}

\begin{document}

\ZTPAuthor{\ZTPHasOrcid{Ksenia Bestuzheva}{0000-0002-7018-7099},
\ZTPHasOrcid{Ambros Gleixner}{0000-0003-0391-5903},
Stefan Vigerske}

\ZTPTitle{A Computational Study of Perspective Cuts}

\ZTPNumber{21-07}
\ZTPMonth{March}
\ZTPYear{2021}

\title{A Computational Study of Perspective Cuts}
\author[1]{Ksenia Bestuzheva}
\author[2]{Ambros Gleixner}
\author[3]{Stefan Vigerske}

\affil[1]{Zuse Institute Berlin, Takustr.~7, 14195 Berlin, Germany\\
          \texttt{bestuzheva@zib.de}}

\affil[2]{Zuse Institute Berlin and HTW Berlin, Germany\\
          \texttt{gleixner@zib.de}}

\affil[3]{GAMS Software GmbH, c/o Zuse Institute Berlin\\
          \texttt{svigerske@gams.com}}

\zibtitlepage
\maketitle

\begin{abstract}
The benefits of cutting planes based on the perspective function are well known
for many specific classes of mixed-integer nonlinear programs with on/off
structures. However, we are not aware of any empirical studies that evaluate
their applicability and computational impact over large, heterogeneous test sets
in general-purpose solvers. This paper provides a detailed computational study
of perspective cuts within a linear programming based branch-and-cut solver for general
mixed-integer nonlinear programs.
Within this study, we extend the applicability of perspective cuts from convex to nonconvex nonlinearities.
This generalization is achieved by applying a perspective strengthening to valid linear inequalities which separate solutions of linear relaxations.
The resulting method can be applied to any constraint where all variables appearing in nonlinear
terms are semi-continuous and depend on at least one common indicator variable.
Our computational experiments show that adding perspective cuts for convex constraints yields a
consistent improvement of performance, and adding perspective cuts for nonconvex constraints
reduces branch-and-bound tree sizes and strengthens the root node relaxation, but has no
significant impact on the overall mean time.

\end{abstract}

\section{Introduction}

Consider a mixed-integer nonlinear program (MINLP) with semi-continuous variables:
\begin{subequations}\label{minlp}
\begin{align}\min\; &\left<\bm c,(\bm x, \bm y, \bm z)\right> \label{eq:obj}\\
\text{s.t.}\; &g(\bm x,\bm y,\bm z) \leq 0,\label{eq:conss}\\
&(\underline{y}_j - y^0_j)z_k \leq y_j - y^0_j \leq (\overline{y}_j - y^0_j)z_k,~\forall j \in \mathcal{S}_k, ~\forall k \in \mathcal{I}, \label{eq:sc} \\
&\bm x \in \rit^n, ~\bm y \in \rit^p, ~\bm z \in \{0,1\}^q.
\end{align}
\end{subequations}
Here, $\left<\bm c,(\bm x,\bm y,\bm z)\right>$ is the linear objective function given by a scalar product of a constant vector $\bm c \in \rit^{n+p+q}$ and the vectors of continuous variables $\bm x$, semi-continuous variables $\bm y$ and binary variables $\bm z$.
Constraints~\eqref{eq:conss} are given by inequalities $g(\bm x,\bm y,\bm z) \leq 0$, where $g: ~\rit^n  \times
\rit^p \times [0,1]^q \rightarrow \rit^m$ is a vector function and some of its elements $g_i$ are nonlinear.

The set~$\mathcal{S}_k \subseteq \{1,\ldots,p\}$ shall contain the indices of semi-continuous
variables controlled by the indicator variable $z_k$, and $\mathcal{I} \subseteq \{1,\dots,q\}$ is
the set of indices of all indicator variables.
Constraints~\eqref{eq:sc} ensure that for each $j \in \mathcal{S}_k$, the value of $y_j$
belongs to the domain $[\underline{y}_j, \overline{y}_j]$ when the indicator variable $z_k$ is
equal to~$1$ and has a fixed value $y^0_j$ when $z_k$ is equal to $0$.
Semi-continuous variables are typically used to model ``on'' and ``off'' states of a process and can be found in such problems as optimal line switching in electrical networks \cite{fisher2008optimal},  blending \cite{williams1978reformulation} and production planning \cite{barbaro1986generalized}, to name but a few.

In order to simplify the notation, in the rest of the paper the subscript $k$ will be omitted and
we will be referring to a vector of semi-continuous variables $\bm y \in \rit^p$ controlled by the
indicator variable $z \in \{0,1\}$.
The semi-continuity relation is then defined by the inequality
$(\underline{\bm y} - \bm y^0)z \leq \bm y - \bm y^0 \leq (\overline{\bm y} - \bm y^0)z.$

Without loss of generality, we consider constraints of the form
\begin{equation}\label{eq:disjcons}g(\bm x,\bm y) = f(\bm y) - x_\ell \leq 0\end{equation}
for some $\ell\in\{1,\ldots,n\}$.
The continuous variable $x_\ell$ represents the linear non-semi-continuous part and
the same arguments as presented in this paper can be directly adapted for a more general linear
part.
When $z = 0$, function $f$ is reduced to a fixed value $f(\bm y^0)$.
A common example of such constraints are on/off constraints which become redundant when the corresponding indicator variable is set to~$0$.

Many state-of-the-art algorithms for the solution of MINLP~\eqref{minlp} make use of nonlinear and
linear programming relaxations where the condition $z \in \{0,1\}$ is replaced with $z \in [0,1]$.
However, for a constraint of the form \eqref{eq:disjcons}, simply dropping the integrality
condition generally does not produce the tightest possible continuous relaxation.
The reason for this is that the dependence of the bounds on $\bm y$ on the
indicator variable $z$ is not exploited by a straightforward continuous relaxation.
Consequently, the same applies for the linearization of Constraint \eqref{eq:disjcons} via gradient
cuts~\cite{kelley1960cutting}, that is, inequalities
\begin{equation}\label{eq:gradientcut}f(\hat{\bm y}) + \left<\nabla f(\hat{\bm y}),\bm y - \hat{\bm y}\right> \leq x_\ell,\end{equation}
where $\hat{\bm y}$ is the point at which $f$ is linearized.

The strongest continuous relaxation of the set described by Constraint~\eqref{eq:disjcons}, given
that $\bm y$ is semi-continuous, can be achieved by applying the perspective
reformulation~\cite{frangioni2006perspective}.
Linearizing this reformulation provides valid linear inequalities known as perspective cuts.

In this paper we present a computational study of perspective cuts within SCIP~\cite{gamrath2020scip},
a general-purpose solver that implements an LP-based branch-and-cut algorithm to solve mixed-integer nonlinear programs to global optimality.
Section~\ref{sec:pcuts} provides the theoretical background for this study and a review of
applications that can be found in existing literature.
In Section \ref{sec:genpcuts}, we describe our approach to creating perspective cuts and show that
for convex instances, it is equivalent to linearizing the perspective formulation via gradient cuts.
Section~\ref{sec:impl} gives an outline of our implementation of perspective cuts in SCIP,
which includes detection of suitable structures and separation and strengthening of perspective cuts.
Finally, in Section \ref{sec:compres} the results of computational experiments on instances from MINLPLib\footnote{\url{https://www.minlplib.org}}~\cite{10.1287/ijoc.15.1.114.15159} are presented.

\section{Perspective Formulations for Convex Nonlinearities}\label{sec:pcuts}

This section gives a review of the existing literature on theoretical and computational results
related to perspective formulations for convex nonlinearities.

\subsection{Theoretical background}\label{sec:theory}

Let $F$ denote the feasible region defined by Constraint \eqref{eq:disjcons} and the semi-continuity
constraints~\eqref{eq:sc} for an indicator $z$.
$F$ can be written as a union of two sets $F^0$ and $F^1$ corresponding to values $0$ and $1$ of $z$, respectively:
\begin{gather}\label{offset}F^0 = \{(x_\ell, \bm y, z) ~|~ x_\ell \geq f(\bm y^0), ~\bm y = \bm y^0, ~z = 0\},\\
\label{onset}F^1 = \{(x_\ell, \bm y, z) ~|~ x_\ell \geq f(\bm y), ~\bm y \in [\underline{\bm y},\overline{\bm y}], ~z = 1\}.\end{gather}

The tightest possible convex relaxation of $F$ is its convex hull. Ceria and Soares \cite{Ceria} studied convex hull formulations for unions of convex sets and their applications to disjunctive programming. Stubbs and Mehrotra \cite{stubbs1999branch} described the convex hull of the feasible set of a convex 0-1 program and developed a procedure for generating cutting planes. Grossmann and Lee \cite{grossmann2003generalized} extended the convex hull results to generalized disjunctive programs (GDPs). Similarly to disjunctive programming, feasible sets of GDPs are given as unions of convex sets, but more general logical relations are also allowed. These works used the perspective function, which is defined as follows:

\begin{definition}\cite{rockafellar2015convex} For a given convex function $f: \rit^p \rightarrow \rit$, its perspective function $\tilde{f}: \rit^{n+1} \rightarrow \left(\rit \cup \{+\infty\}\right)$ is defined as:
	$$\tilde{f} (\bm y,z) =
	\begin{cases}
	zf(\bm y/z), & \text{ if } z>0,\\
	+\infty, & \text{ otherwise,}
	\end{cases}
	$$
	where $\bm y \in \rit^p$, $z \in \rit$.
\end{definition}

These early results are applicable to convex sets with few non-restrictive conditions and utilize an
extended variable space.
Frangioni and Gentile~\cite{frangioni2006perspective} proposed a
reformulation in the original space for a special case.
Considering a semi-continuous vector $\bm y$, an
indicator variable $z$, and a convex function $f$ depending only on $\bm y$ such that
$\bm y^0 = (0,\dots,0)$ and $f(\bm 0) = 0$, they capture the disjunctive structure by defining a new
nonconvex function $f^d$:

$$f^d(\bm y,z) = \begin{cases} f(\bm y) &\text{ if } z = 1, ~y \in [\underline{\bm y}, \overline{\bm y}],\\
                         0 &\text{ if } \bm y = z = 0,\\
                         +\infty &\text{ otherwise.}\end{cases}$$
The function $f^d$ is directly related to the set $F$: the latter can be described as the set of all
points $(x_\ell,\bm y,z)$ such that $f^d(\bm y,z)$ is finite and $f^d(\bm y,z) \leq x_\ell$.
Frangioni and Gentile describe the convex envelope of $f^d$:
$$\overline{co}f^d(\bm y,z) = \begin{cases}\tilde f(\bm y,z) &\text{ if }z \in (0,1],\\
                                     0 &\text{ if } z = 0,\\
                                     +\infty &\text{ otherwise.}\end{cases}$$

In a related work, G{\"u}nl{\"u}k and Linderoth~\cite{Gunluk10} show that the convex hull of $F$ is
given by
\begin{align}
  \text{conv}(F) = \{(x_\ell,\bm y,z) ~|~ \overline{co}f^d(\bm y,z) &\leq x_\ell,\nonumber\\
  (\underline{\bm y} - \bm y^0)z &\leq \bm y -
  \bm y^0 \leq (\overline{\bm y} - \bm y^0)z, ~z \in [0,1]\}.\label{eq:pr}
\end{align}

Therefore, replacing $f$ with $\overline{co}f^d$ in Constraint~\eqref{eq:disjcons} results in a
reformulation with the tightest possible continuous relaxation: the \emph{perspective reformulation}.  This is a valid reformulation
since $f = \overline{co}f^d$ for $z \in \{0,1\}$.

\begin{figure*}[!ht]
	\begin{subfigure}{.49\linewidth}
	  \centering
	  \includegraphics[width=1\linewidth]{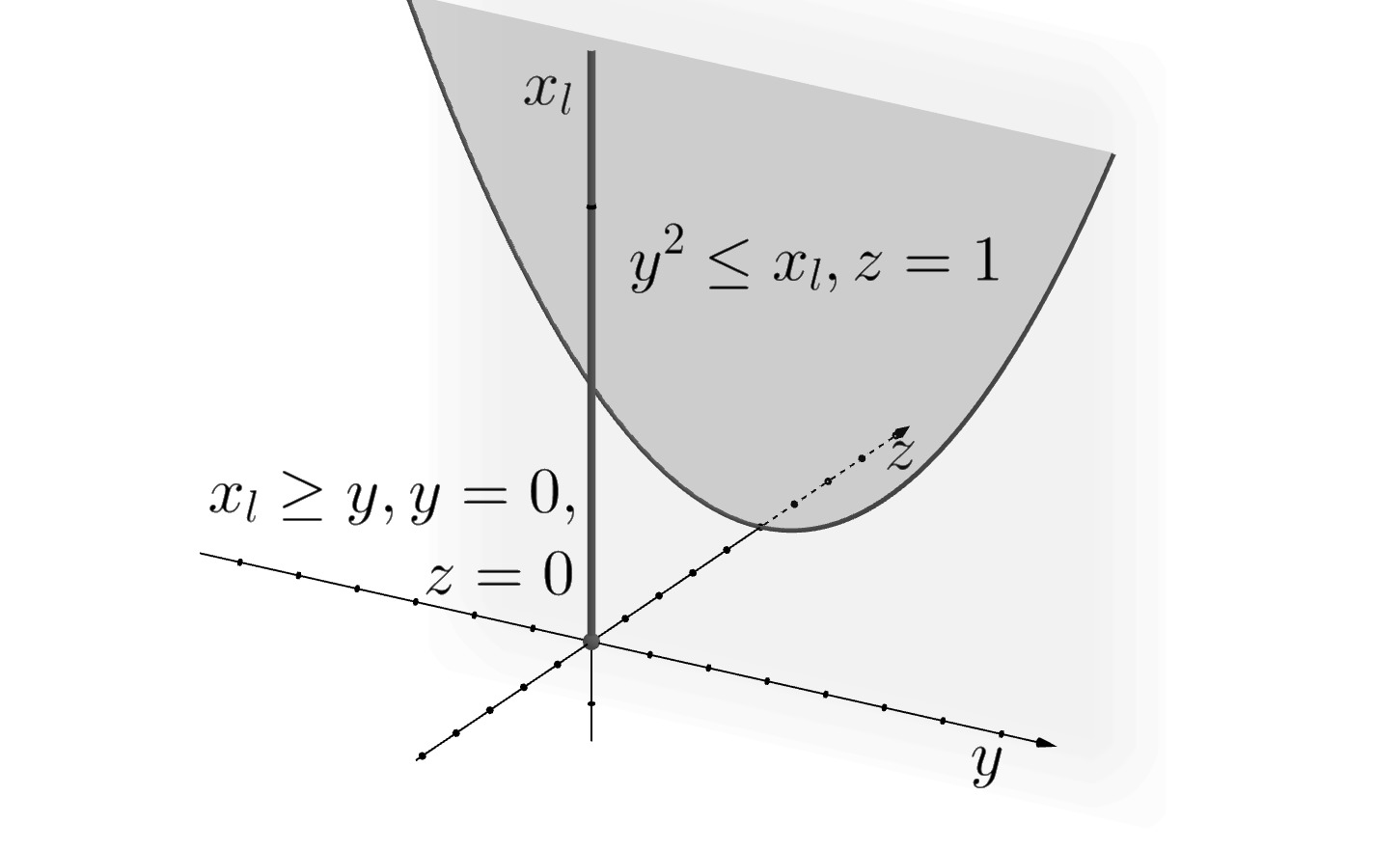}
	  \caption{the mixed-integer set}
          \label{fig:disj_set_ex0}
	\end{subfigure}
	\begin{subfigure}{.49\linewidth}
	  \centering
	  \includegraphics[width=1\linewidth]{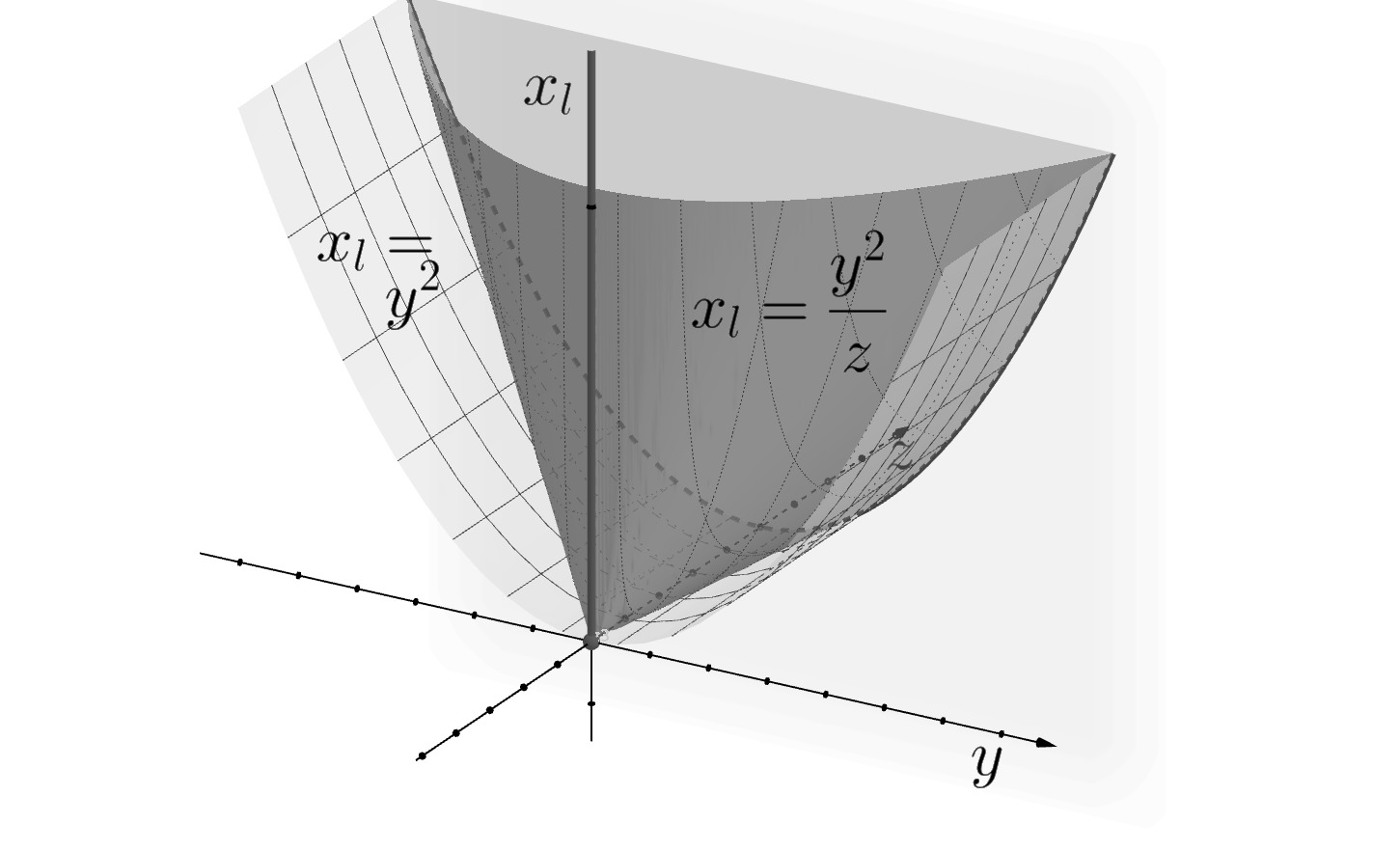}
	  \caption{the continuous relaxations}
          \label{fig:disj_set_ex1}
	\end{subfigure}
	\caption{Example of a disjunctive set and its convex hull}
	\label{fig:disj_set_ex}
\end{figure*}

Figure~\ref{fig:disj_set_ex} shows an example of a disjunctive set and compares its continuous
relaxations.
The disjunctive set (shown in Figure~\ref{fig:disj_set_ex0}) consists of the ray $\{(x_\ell,y,z) ~|~ x_\ell \geq 0, ~y = 0,
~z = 0\}$ and the convex set $\{(x_\ell,y,z) ~|~ x_\ell \geq f(y), ~z = 1\}$, where $f(y)=y^2$.
The convex hull, shown in Figure~\ref{fig:disj_set_ex1}, is the closure of the set of all points
above the dark gray surface defined by $x_\ell = y^2/z$, $z \in (0,1]$.
This equation is obtained by applying the perspective operator to $f$:
$\tilde f(y,z) = zf(y/z) = y^2/z$.
For comparison, the boundary of the straightforward continuous relaxation given by $x_\ell = y^2$,
$z \in [0,1]$, is shown in Figure~\ref{fig:disj_set_ex1} in light gray color.

Due to the division by $z$ in the perspective function, the perspective reformulation~\eqref{eq:pr}
is non-differentiable at $z = 0$.
In some special cases, formulation~\eqref{eq:pr} can be written as a second-order cone (SOC)
constraint~\cite{Tawarmalani00convexextensions,AKTURK2009187,Gunluk10,FRANGIONI2009206}.
In particular, this is possible when Constraint~\eqref{eq:disjcons} itself is SOC-representable.

Frangioni and Gentile~\cite{frangioni2011projected,frangioni2016approximated} introduced projection
approaches for additively separable closed convex functions.
In the \emph{projected perspective reformulation} (P$^2$R)~\cite{frangioni2011projected}, the perspective
function is projected into the space of continuous variables and rewritten as a piecewise-convex
function.
This technique avoids the numerical issues associated with the perspective functions while yielding
strong bounds, but at the cost of using piecewise-continuous functions which cannot be directly
passed to off-the-shelf solvers.
The \emph{approximated projected perspective reformulation} (AP$^2$R)
method~\cite{frangioni2016approximated} lifts the P$^2$R formulation back into the original space
by reintroducing the indicator variables.
AP$^2$R can be solved by general-purpose solvers and, if $\bm y^0 \leq \underline{\bm y}$, has the same
number of variables and constraints as the original problem.
The bound provided by AP$^2$R is, however, generally weaker than the one from P$^2$R.

Alternatively, the perspective reformulation~\eqref{eq:pr} can be represented by an infinite number
of linear outer approximations which are then dynamically separated.
Suppose that we have a point $(\hat x_\ell, \hat{\bm y}, \hat z)$ such that $\hat z \in (0,1)$ and
$\hat x_\ell < \tilde f(\hat{\bm y}, \hat z)$.
By performing first-order analysis of the convex envelope $\overline{co}f^d$, Frangioni and
Gentile~\cite{frangioni2006perspective} derive cuts that separate $(\hat x_\ell, \hat{\bm y}, \hat z)$
from the convex hull of $F$, referred to as \emph{perspective cuts}:
\begin{equation}\label{eq:pc0}\left<\nabla f(\bm y^*),\bm y\right> + \left(f(\bm y^*) - \left<\nabla f(\bm y^*),\bm y^*\right>\right)z \leq x_\ell,\end{equation}
where $\bm y^* = \hat{\bm y} / \hat z$.

It is easy to adjust the perspective reformulation and the inequalities~\eqref{eq:pc0} for the case
of nonzero $\bm y^0$ and $f(\bm y^0)$:
\begin{equation}\label{eq:pc}\left<\nabla f(\bm y^*),\bm y-\bm y^0\right> + (f(\bm y^*) - f(\bm y^0) - \left<\nabla f(\bm y^*),\bm y^* - \bm y^0\right>)z + f(\bm y^0) \leq x_\ell,\end{equation}
where $\bm y^* = (\hat{\bm y} - \bm y^0) / \hat z + \bm y^0$.
We refer to~\eqref{eq:pc} as the \emph{perspective cut} at $\bm y^*$.

\subsection{Existing applications and computational results}

Perspective cuts and reformulations were tested on several applications which contain convex functions of semi-continuous variables.

Frangioni and Gentile~\cite{frangioni2006perspective} applied perspective cuts~\eqref{eq:pc} to the \emph{thermal unit commitment} problem.
In order to avoid the technical difficulties of incorporating perspective cuts into a general-purpose solver, the authors implemented their own NLP-based branch-and-cut algorithm.
Perspective cuts are applied to the objective function via a specialized separation procedure,
which replaces a univariate term with its perspective linearization if the perspective linearization
is tighter at the current relaxation solution.
The linearization is represented by an auxiliary variable, and as more perspective cuts are added
for the term, the variable is set to be equal to the maximum of all linearizations.
Perspective cuts were shown to have a considerable impact on the performance.
The geometric mean of the running time of the best performing setting was smaller than that for the algorithm with the straightforward continuous relaxation by a factor of $60$.

Perspective reformulations were studied by G{\"u}nl{\"u}k and
Linderoth~\cite{Gunluk10,10.1007/978-1-4614-1927-3_3}.
Their key observation is that for some problems, the perspective reformulation can be written with the use of second-order cone constraints.
The applications studied in this paper are:
\begin{itemize}
\item \emph{Separable quadratic uncapacitated facility location} on a testset consisting of $16$ instances.
With the perspective reformulation, $50\%$ more instances are solved within the time limit of $8$ hours and on the instances that are solved with both formulations, the perspective formulation is faster by a factor of $8$ when comparing the geometric mean.
\item \emph{Network design with congestion constraints} on a testset consisting of $35$ instances.
The perspective formulation is solved for 29 instances within the time limit of $4$ hours, as opposed to only 2 instances with the standard formulation.
\item \emph{Mean-variance optimization (portfolio optimization)} on a testset consisting of $20$ instances.
Although none of the instances are solved within the time limit of 10,000 CPU seconds, perspective reformulation significantly improves the gap.
For example, the final gap between the best found lower and upper bounds is reduced from $185.1\%$
with the standard formulation to $4.2\%$ with the perspective reformulation on
instances of smaller size, and from $490.0\%$ to $5.9\%$ on instances of larger size.
\end{itemize}

Atamt{\"u}rk and G{\'o}mez~\cite{atamturk2018strong} applied the perspective-based conic reformulation to the \emph{image segmentation} problem, testing it on $4$ instances of different sizes.
On the one instance that was solved within a time limit of $1$ hour, the running time was reduced by a factor of $18$ when compared to the standard formulation.
On the three remaining instances, using the perspective formulation resulted in a $45$-$55\%$ decrease of the remaining gap at time limit.

Aktürk et al.~\cite{AKTURK2009187} presented a perspective-based conic reformulation of the
\emph{machine-job assignment problem with controllable processing times}.
The tests were conducted on 180 randomly generated instances of varying sizes with quadratic and cubic
objectives.
For problems with a quadratic objective, 91\% of the 90 instances were solved
when using the strengthened conic formulation, whereas at most 36\% of instances were solved when
using non-perspective formulations.
For problems with a cubic objective, 88\% of the 90 instances were solved with the perspective formulation
and at most 27\% were solved with non-perspective formulations.

A comparison between SOC-based perspective formulations and perspective cutting planes was performed by
Frangioni and Gentile~\cite{FRANGIONI2009206}.
Using the CPLEX-11 solver, they test the two approaches on two sets of mixed-integer quadratic
problems, namely, the \emph{Markowitz mean-variance model} and the \emph{unit commitment problem}.
The difference between the two formulations is particularly significant with the setup used in the
paper~\cite{FRANGIONI2009206} since by default, CPLEX obtains dual bounds by solving nonlinear relaxations.
The results favor the cutting planes approach, with the difference being larger for the
Markowitz mean-variance problem.
The authors observe that the advantage of perspective cuts stems mostly from efficient
reoptimization of linear programs.
They add that the perspective conic reformulation is more competitive for problems that are larger,
more nonlinear (i.e., have more nonlinear constraints or non-quadratic nonlinear constraints) or
have richer structure.

Frangioni and Gentile~\cite{frangioni2011projected,frangioni2016approximated} tested the
projected perspective reformulation (P$^2$R) and the approximated perspective projected reformulation
(AP$^2$R) on \emph{sensor placement}, \emph{nonlinear network design},
\emph{mean-variance portfolio} and \emph{unit commitment} problems.
P$^2$R was implemented as part of a specialized branch-and-bound algorithm and AP$^2$R was solved
directly with CPLEX 12.
Both approaches were compared to perspective cuts implemented as a callback in CPLEX.
Computational results show that P$^2$R is the best performing method for problems that have
a well-suited structure and require little or no branching.
For problems with more complex structures, AP$^2$R is competitive with the perspective cut approach.
When there are constraints linking indicator variables and few linear approximations provide a
good estimate of the original nonlinear function, perspective cuts tend to be the best performing method.

Salgado et al.~\cite{salgado2018perspective} studied the \emph{alternating current optimal power
flow problem with activation/deactivation of generators} (ACOPFG) using $8$ test instances.
They tested two perspective-based reformulations of the objective function.
The first uses four perspective cuts of the form~\eqref{eq:pc}; the other is obtained by applying
AP$^2$R~\cite{frangioni2016approximated}.
Although the results of enhancing the standard ACOPFG model with perspective reformulations are
inconclusive, an outer approximation~\cite{salgado2018alternating} of the problem significantly
benefits from both perspective cuts and AP$^2$R.
The perspective cuts approach performs best, solving one more instance than the standard
formulation within the time limit of $1$~hour and taking less than $4$ seconds on all the remaining
instances, whereas the standard formulation requires over $1000$ seconds on most instances.

\section{Generalized perspective cuts}\label{sec:genpcuts}

If $f$ is non-convex, neither the gradient cuts~\eqref{eq:gradientcut} nor the perspective
cuts~\eqref{eq:pc} are guaranteed to be valid.
However, the perspective reformulation can be applied to a convex underestimator of $f$, from which
the perspective cuts~\eqref{eq:pc} can be derived.
Alternatively, any linear inequality $\phi(\bm y) \leq x_\ell$ that is valid for the `on' set $F^1$
can be adjusted for the `off' set $F^0$.

In the following,
we propose a cut extension procedure that ensures that the generated inequality is equivalent to
$\phi(\bm y) \leq x_\ell$ when the indicator is equal to $1$ and holds with equality at the point
$(x_\ell,\bm y,z) = (f(\bm y^0), \bm y^0, 0)$.
\begin{theorem}[Generalized perspective cuts]\label{thm:cut_strengthening}
Consider a vector of semi-continuous variables $\bm y\in\mathbb{R}^p$ with an indicator $z\in\{0,1\}$, such that $\bm y = \bm y^0$ if $z = 0$, and a
linear inequality $\phi(\bm y) \leq x_\ell$ that is valid for the set
$$F^1 = \{(x_\ell, \bm y, z) ~|~ x_\ell \geq f(\bm y), ~\bm y \in [\underline{\bm y},\overline{\bm y}], ~z = 1\}.$$
Let
$$\tilde \phi(\bm y,z) = \phi(\bm y) + \left(f(\bm y^0) - \phi(\bm y^0)\right)(1-z).$$
Then the linear inequality $\tilde \phi(\bm y,z) \leq x_\ell$ is valid for the set $F^0 \cup F^1$, where
$$F^0 = \{(x_\ell, \bm y, z) ~|~ ~x_\ell \geq f(\bm y^0), ~\bm y = \bm y^0, ~z = 0\}.$$
\end{theorem}
\begin{proof}
It is sufficient to check the validity for each possible value of $z \in \{0,1\}$.
By substituting $z = 1$ and $z = 0$ in $\tilde \phi(\bm y,z)$, we immediately obtain
\begin{enumerate}
\item $\tilde \phi(\bm y,1) = \phi(\bm y) ~\forall \bm y \in \mathbb{R}^p$ and

\item $\tilde \phi(\bm y^0,0) = f(\bm y^0)$,
\end{enumerate}
respectively. Therefore, $\tilde \phi(\bm y,z) \leq x_\ell$ is a valid inequality.
\qed
\end{proof}

If the cut $\phi(\bm y) \leq x_\ell$ is already valid for $F^0$, then the described above adjustment always
produces a cut that is at least as strong as the original cut.
Since $\phi(\bm y) \leq x_\ell$ is in this case implied by $f(\bm y) \leq x_\ell$ for
$(x_\ell,\bm y) \in F^0$, we have $\phi(\bm y^0) \leq f(\bm y^0)$.
Hence the coefficient of $(1-z)$ in $\tilde \phi(\bm y,z)$ is nonnegative and
$$\tilde \phi(\bm y,z) \geq \phi(\bm y), ~\forall z \in [0,1], ~\forall \bm y\in\mathbb{R}^p.$$
If additionally $\phi(\bm y^0) < f(\bm y^0)$, i.e., the original cut is not tight at $\bm y^0$, then the new
cut is also stronger.
Otherwise, if $\phi(\bm y) \leq x_\ell$ does not hold for $F^0$ (that is, if $\phi(\bm y^0) > f(\bm y^0)$), then the adjustment is necessary to obtain a cut
that is valid for $F^0 \cup F^1$.

This cut extension procedure has two main advantages:
\begin{enumerate}
\item It does not depend on the convexity of $f$ and requires no assumptions
on the cut except for its validity for $F^1$.
\item In the case where $\bm y^0 \notin [\underline{\bm y}, \overline{\bm y}]$, variable bounds for $F^1$ are
tighter than those for $F^0 \cup F^1$.
This is useful for non-convex constraints since the tightness of their relaxations depends on
variable bounds, and therefore cuts constructed for $\bm y \in [\underline{\bm y}, \overline{\bm y}]$ will
generally be stronger than those for $\bm y \in [\min\{\bm y^0,\underline{\bm y}\}, \max\{\bm y^0,\overline{\bm y}\}]$.
\end{enumerate}

When the cut strengthening is applied to the convex setting, the result is equivalent to the
well-known perspective cuts:

\begin{theorem}[Alternative derivation of perspective cuts]\label{prop1}
Suppose that $f: \rit^p \rightarrow \rit$ is convex and
$(\hat x_\ell, \hat{\bm y}, \hat z) \not\in \text{conv}(F)$ as defined in Section~\ref{sec:theory}.
Consider the gradient cut~\eqref{eq:gradientcut} at point $\bm y^* = (\hat{\bm y} - \bm y^0) /
\hat z + \bm y^0$ for Constraint~\eqref{eq:disjcons}:
$$\phi(\bm y) = f(\bm y^*) + \left<\nabla f(\bm y^*),\bm y - \bm y^*\right> \leq x_\ell.$$
Let $\tilde \phi(\bm y,z)$ be the linear function obtained from $\phi(\bm y)$ by following the strengthening procedure
in Theorem~\ref{thm:cut_strengthening}.
Then $\tilde \phi(\bm y,z)$ is written as follows:
$$\tilde \phi(\bm y,z) = \left<\nabla f(\bm y^*),\bm y-\bm y^0\right> + (f(\bm y^*) - f(\bm y^0) - \left<\nabla f(\bm y^*),\bm y^* - \bm y^0\right>)z + f(\bm y^0)$$
and the cut $\tilde \phi(\bm y,z) \leq x_\ell$ is equivalent to the perspective cut \eqref{eq:pc} at point $(\hat{\bm y},\hat z)$.
\end{theorem}

\begin{proof}
The coefficient of $(1-z)$ in $\tilde \phi(\bm y,z)$ is
$$\alpha = f(\bm y^0) - \phi(\bm y^0) = f(\bm y^0) - f(\bm y^*) - \left<\nabla f(\bm y^*),\bm y^0 - \bm y^*\right>.$$

Adding $\alpha(1-z)$ to the left hand side of the gradient cut produces the perspective cut \eqref{eq:pc}:
$$\tilde \phi(\bm y,z) = \phi(\bm y) + \alpha(1-z) =$$
$$\left<\nabla f(\bm y^*),\bm y-\bm y^0\right> + (f(\bm y^*) - f(\bm y^0) - \left<\nabla f(\bm y^*),\bm y^* - \bm y^0\right>)z + f(\bm y^0).$$
\qed
\end{proof}

To paraphrase, for a convex function $f$ the perspective cut at a solution
$(\hat x_\ell, \hat{\bm y}, \hat z)$ of the LP relaxation can equivalently be obtained
by first generating a gradient cut for $f$ at the modified point $\bm y^*$ and then
applying the strengthening procedure from
Theorem~\ref{thm:cut_strengthening}.

Let us consider an example to illustrate the cut extension method.

\begin{example}\label{example}
Consider a constraint $f(y) = -y^3 + y \leq x_\ell$.
The boundary of the feasible region is shown in Figure~\ref{fig:cut_str_ex} by the dark gray nonlinear surface,
and the feasible points are located above it.
Let $0.5z \leq y \leq z$, where $y$ is a scalar, semi-continuous variable modeled using the binary variable $z \in\{0,1\}$.

\begin{figure}[!ht]
	\begin{subfigure}{.49\linewidth}
	  \centering
	  \includegraphics[width=1\linewidth]{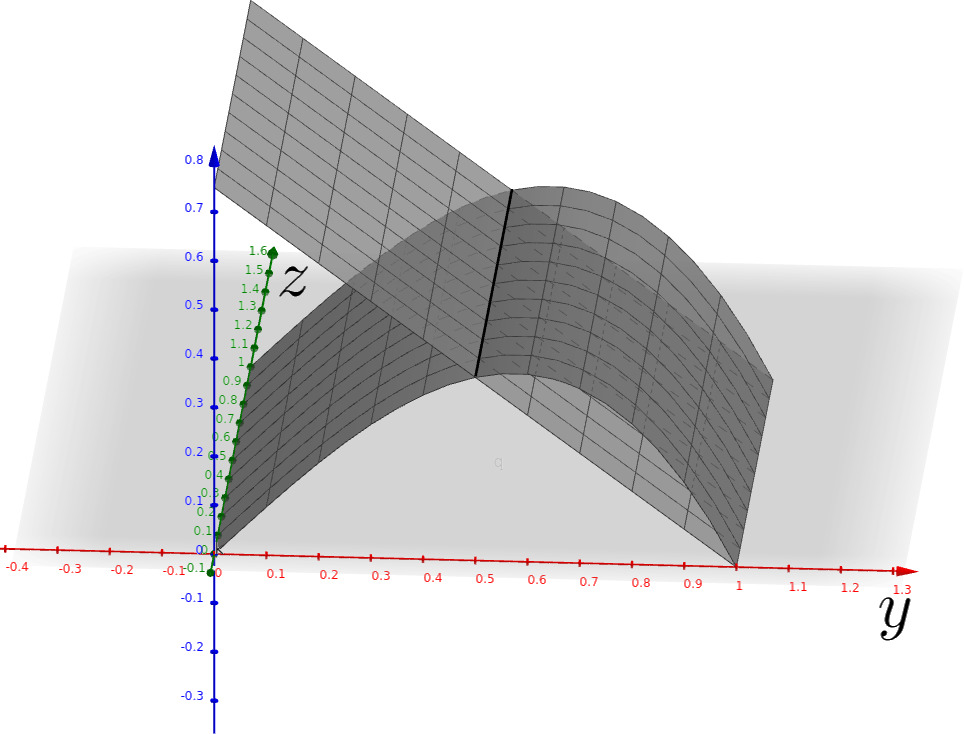}
	  \caption{original cut valid for $z=1$, $y \in [0.5,1]$}
          \label{fig:cut_str_ex0}
	\end{subfigure}
	\begin{subfigure}{.49\linewidth}
	  \centering
	  \includegraphics[width=1\linewidth]{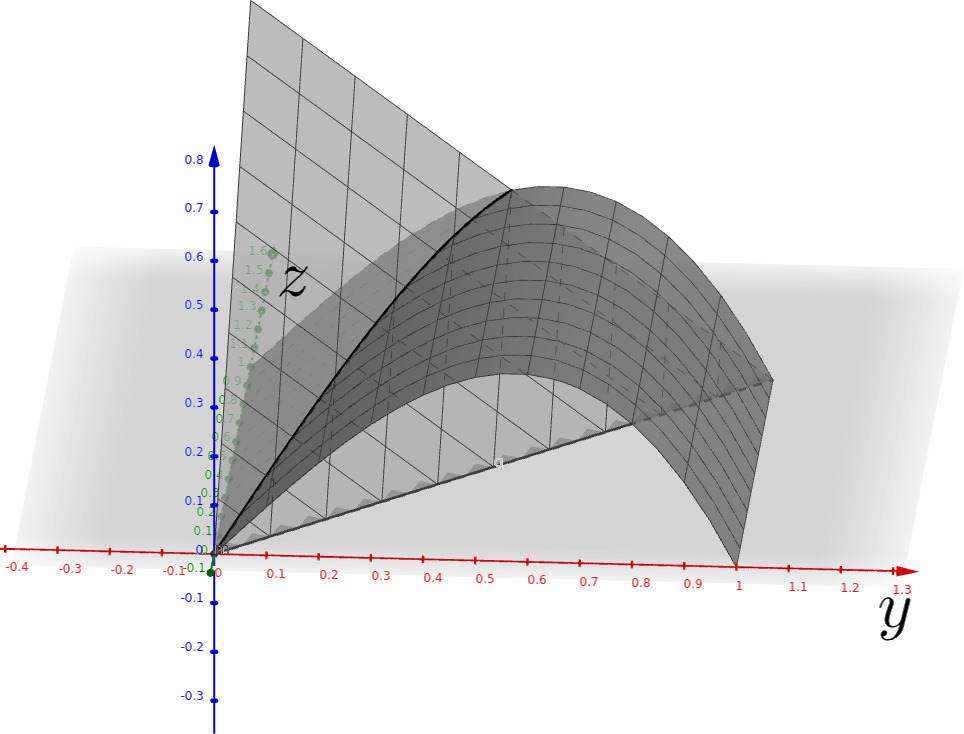}
	  \caption{generalized perspective cut}
          \label{fig:cut_str_ex1}
	\end{subfigure}
	\caption{Example of cut extension}
	\label{fig:cut_str_ex}
\end{figure}

First we find an underestimator of $f(y)$ valid for $z=1$.
In this case, $y$ is constrained to belong to the interval $[0.5,1]$.
Since $f(y)$ is concave on $[0.5,1]$, the underestimator is the secant through points $(0.5,f(0.5))$
and $(1,f(1))$:
$$f^{sec}(y) = -0.75y + 0.75.$$
The cut $f^{sec} \leq x_\ell$ (shown in Figure~\ref{fig:cut_str_ex0}) is not valid for the
whole feasible set.
In particular, a feasible point $(y,z,x_\ell) = (0,0,0)$ violates the cut:
$f^{sec}(0,0) = 0.75 > 0 = x_\ell$.

Now we extend the cut so that to ensure validity at $z=0$.
By Theorem~\ref{thm:cut_strengthening}, the new cut is written as:
$$\tilde f^{sec} = f^{sec} + (f(y^0) - f^{sec}(y^0))(1-z) = -0.75y + 0.75z.$$
This cut, shown in Figure~\ref{fig:cut_str_ex1}, is valid for the whole feasible set given
by the cubic constraint and the semi-continuity condition.
\end{example}

\section{Implementation of perspective cuts}\label{sec:impl}

An effective implementation of perspective cuts within a general-purpose solver requires providing methods for
detecting suitable structures in a general problem and generating the cuts during the solution process.
In the following, we describe our implementation within SCIP, but many considerations discussed here will be
applicable to MINLP solvers in general.

\subsection{Organization of nonlinear constraints in SCIP}

SCIP builds a relaxation for the MINLP \eqref{minlp} by means of an extended formulation, where
auxiliary variables $\bm w$ are introduced for the subexpressions that constitute the constraint
functions $g(\bm x,\bm y,\bm z)$.
Without loss of generality, we can assume that \eqref{eq:conss} has been replaced by a new system
\begin{equation}\label{eq:extform}
  \begin{aligned}
 &h_i(\bm x,\bm y,w_1,\ldots,w_{i-1},\bm z) \lesseqqgtr w_i, \qquad i = 1,\ldots,m',\\
 &\bm w^{\ell} \leq \bm w \leq \bm w^u,
 \end{aligned}
\end{equation}
where $\bm w^l$ and $\bm w^u$ denote global lower and upper bounds on $\bm w$.

The handling of nonlinear constraints in the version of SCIP used for this work is
performed by modules called ``nonlinearity handlers''.
Each nonlinearity handler works on a specific structure (e.g. quadratic, convex, etc.) and provides
callback methods.
For the purposes of this paper, three types of callbacks are relevant:
\begin{itemize}
\item \emph{Detection callbacks} receive an expression and determine whether it is suitable for the
nonlinearity handler.
\item \emph{Estimation callbacks} provide linear under- and overestimators given an expression and a
point at which to linearize it.
\item \emph{Enforcement callbacks} enforce a given violated constraint by adding cutting planes,
tightening bounds, detecting infeasibility, etc.
\end{itemize}

Our addition of generalized perspective cuts is implemented via a specialized
perspective nonlinearity handler.

\subsection{Structure detection}\label{subsection:detection}

The detection algorithm identifies constraints of the form \eqref{eq:extform}, where $h_i$ is
nonlinear and at least one other nonlinearity handler provides an estimation callback for it.
All variables that $h_{i}$ depends on must be semi-continuous with at least one common indicator
variable.
If several binary variables satisfying this condition are found, all such variables are stored for
use in cut generation.

A special case is that of $h_i$ being a sum.
Here, only the variables appearing in nonlinear terms of the sum are required to be
semi-continuous.

To determine whether a variable $y_j$ is semi-continuous, the detection callback of the perspective
nonlinearity handler searches for pairs of implied bounds on $y_j$ with the same indicator $z_k$:
\begin{align*}
y_j &\leq \alpha^{(u)} z_k + \beta^{(u)},\\
y_j &\geq \alpha^{(\ell)} z_k + \beta^{(\ell)}.
\end{align*}
If $\beta^{(u)} = \beta^{(\ell)}$, then $y_j$ is a semi-continuous variable and $y_j^0 = \beta^{(u)}$,
$\underline{y}_j = \alpha^{(\ell)} + \beta^{(\ell)}$ and $\overline{y}_j = \alpha^{(u)} + \beta^{(u)}$.

This information can be obtained either directly from linear constraints in $y_j$ and $z_k$, or by
finding implicit relations between $y_j$ and $z_k$.
Such relations can be detected by probing, which fixes $z_k$ to its possible values and propagates
all constraints in the problem, thus detecting implications of $z_k = 0$ and $z_k = 1$.
SCIP stores the implied bounds in a globally available data structure.

In addition, the perspective nonlinearity handler detects semi-continuous auxiliary
variables, that is, variables $w_i$ that were introduced to express the extended
formulation~\eqref{eq:extform}.
Given $h_i(\bm y, w_1,\ldots,w_{i-1}) \lesseqqgtr w_i$, where $\bm y,w_1,\ldots,w_{i-1}$ are semi-continuous variables
depending on the same indicator $z_k$, the auxiliary variable $w_i$ is semi-continuous with
$w_i^0 = h_i(\bm y^0,w^0_1,\ldots,w^0_{i-1})$ and $[\underline{w}_i,\overline{w}_i] = h_i([\underline{\bm y}^0,\overline{\bm y}^0],[\underline{w}_1,\overline{w}_1],\ldots,[\underline{w}_{i-1},\overline{w}_{i-1}])$ computed by interval arithmetics.

According to Theorem~\ref{thm:cut_strengthening}, the constraint must have the form
$$w_i \geq h_i(\bm y,w_1,\ldots,w_{i-1}),$$
where all variables $\bm y, w_1,\ldots, w_{i-1}$ are semi-continuous with respect to the same indicator
$z_k$.
In our implementation we allow a more general form:
\begin{align}h_i(\bm x,\bm y,w_1,\ldots,w_{i-1},\bm z) &= h^{sc}_{i,k}(\bm y,w_1,\ldots,w_r) \label{eq:extform_gen}\\
&+ h^{nsc}_{i,k}(\bm x,w_{r+1},\ldots,w_{i-1},\bm z) \lesseqqgtr w_i, \nonumber\end{align}
where the variables are assumed to be sorted so that semi-continuous auxiliary variables
$w_1,\ldots,w_r$ come before the non-semi-continuous auxiliary variables $w_{r+1},\ldots,w_{i-1}$
and the auxiliary variable $w_i$ representing $h_i$.

Thus, for each suitable indicator $z_k$ the function
$h_i$ is split up into the semi-continuous part $h^{sc}_{i,k}$, which can depend only on variables
that are semi-continuous with respect to indicator $z_k$, and a non-semi-continuous
part $h^{nsc}_{i,k}$, which depends on non-semi-continuous variables.
The non-semi-continuous part must be linear.
If the sum has a constant term, the constant is considered to be part of $h^{sc}_{i,k}$.

\subsection{Separation and strengthening of generalized perspective cuts}\label{subsection:separation}

During the cut generation loop, generalized perspective
cuts as in Theorem~\ref{thm:cut_strengthening} are constructed for
constraints of the form~\eqref{eq:extform_gen}.

In the following, let $\bm v$ denote the vector of all problem variables: $\bm v = (\bm x,\bm y,\bm w,\bm z)$,
and let $F^1_{i,k}$ and $F^0_{i,k}$ be the sets of points satisfying a constraint of the form~\eqref{eq:extform_gen} for a
given $i\in\{1,\ldots,m'\}$ together with implied variable bounds for $z_k=1$ and $z_k=0$, respectively.
For simplicity, we fix the inequality sign and consider ``less than or equal to'' constraints:
\begin{align*}
F^1_i = &\{\bm v \,|\, h_i(\bm x,\bm y,w_1,\ldots,w_{i-1},\bm z) \leq w_i, \bm y \in [\underline{\bm y}, \overline{\bm y}], \bm w \in [\underline{\bm w}, \overline{\bm w}], z_k=1\}, \\
F^0_i = &\{\bm v \,|\, h_i(\bm x,\bm y,w_1,\ldots,w_{i-1},\bm z) \leq w_i, \bm y = \bm y^0, \bm w = \bm w^0, z_k=0\}.
\end{align*}

Suppose that the point $\hat{\bm v}$ violates the nonlinear constraint:
$$h_i(\hat{\bm x},\hat{\bm y},\hat{w}_1,\ldots,\hat{w}_{i-1},\hat{\bm z}) > \hat{w}_i.$$
If $h_i$ is nonconvex and its estimators depend on variable bounds, the enforcement callback
first performs probing for $z_k=1$ in order to tighten the implied bounds $[\underline{\bm y}, \overline{\bm y}]$
and $[\underline{w}_j, \overline{w}_j]$ for $j \leq r$.

Estimation callbacks of non-perspective nonlinearity handlers are called in order to find valid cuts
that separate $\hat{\bm v}$ from $F^1_{i,k}$, which are then modified according to Theorem~\ref{thm:cut_strengthening}.
For a constraint of the generalized form $h_i = h^{sc}_{i,k} + h^{nsc}_{i,k} \leq w_i$, which is described in
Section~\ref{subsection:detection}, an estimation callback will provide an underestimator of $h_i$:
$$\underline{h}_i = \underline{h}^{sc}_{i,k} + h^{nsc}_{i,k}.$$

This underestimator consists of an underestimator of the semi-continuous part $\underline{h}^{sc}_{i,k}$
and the non-semi-continuous part $h^{nsc}_{i,k}$, which remains unchanged since it is already linear
and shares none of the variables with the semi-continuous part.
The extension procedure from Theorem~\ref{thm:cut_strengthening} is applied only to
$\underline{h}^{sc}_{i,k}$ to obtain $\underline{\tilde h}^{sc}_{i,k}$.
Since $h^{sc}_{i,k}$ depends only on semi-continuous variables, similar arguments to
Section~\ref{sec:genpcuts} hold for feasibility and tightness of the strengthened underestimator
$\underline{\tilde h}^{sc}_{i,k}$.
The strengthened underestimator of $h_i$ is then written as
$$\underline{\tilde h}_{i,k} = \underline{\tilde h}^{sc}_{i,k} + h^{nsc}_{i,k}.$$
If $\hat{\bm v}$ violates the cut $\underline{\tilde h}_{i,k} \leq w_i$, the cut is passed to the SCIP core
where it will be considered for addition to the LP relaxation.

Let us consider an example of generalized perspective cut separation by extending
Example~\ref{example}.

\begin{example}
The extended formulation of the constraint from Example~\ref{example} is written as:
$$h(x_l,y) = -y^3 + y - x_l \leq w \leq 0.$$

The semi-continuous part of $h$ is $h^{sc}(y) = -y^3 + y$, and the non-semi-continuous part is
$h^{nsc}(x_l) = -x_l$.
The perspective underestimator of $h$ is $\underline{h}(x_l,y,z) = -0.75y + 0.75z - x_l$ and the perspective
cut is written as follows:
$$-0.75y + 0.75z - x_l \leq w.$$

Suppose that the point to be separated is $(\hat x_l,\hat y, \hat z,\hat w) = (0,0.4,0.7,0)$.
Substituting the variables with their values at this point in the perspective cut, we get a
violated inequality $0.225 \leq 0$.
Therefore the cut is violated and will be considered for addition to the LP relaxation.
\end{example}

\section{Computational results}\label{sec:compres}

This section presents the results of computational experiments.
A development version of SCIP (githash \texttt{f0ee1d793d}) was used, together with the linear
solver SoPlex 5.0.1.3~\cite{gamrath2020scip} and the nonlinear solver
Ipopt 3.12.13~\cite{wachter2006implementation}.
All the experiments were run on a cluster of 3.60GHz Intel Xeon E5-2680 processors with 64 GB memory per node.
The time limit was set to one hour and the optimality gap limit to 0.01\%.

Throughout the section, we analyze the following settings,
each defined by the types of constraints for which perspective cuts are added:

\begin{itemize}
\item \emph{Off}: perspective cuts are disabled;

\item \emph{Convex}: perspective cuts are enabled only for convex constraints;

\item \emph{Full}: perspective cuts are enabled for both convex and nonconvex constraints.
\end{itemize}

\subsection{Detection of suitable structures}

Out of the 1703 instances of MINLPLib, suitable constraints of the extended form~\eqref{eq:extform_gen}
were detected for 186 instances.
Table~\ref{tab:detection} shows the numbers of instances where at least one such constraint was
detected, when counting: all instances, instances where detection succeeded for convex constraints
only, instances where detection succeeded both for convex and nonconvex constraints and instances
where detection succeeded for nonconvex constraints only.

\begin{table}[!ht]
   \small
   \caption{Detection results}
   \centering
   \label{tab:detection}
   \begin{tabular}{l @{~~} | >{\centering}p{2cm} | >{\centering}p{2cm} | >{\centering\arraybackslash}p{2cm} }
      \hline\noalign{\smallskip}
        All  &  Convex  &  Both  &  Nonconvex \\
      \noalign{\smallskip}\hline\noalign{\smallskip}
        186  &  89      &  53    &  44        \\
      \noalign{\smallskip}\hline
   \end{tabular}
\end{table}

Only those instances were counted for Table~\ref{tab:detection} where suitable constraints
were detected in the main problem.
Additionally, sometimes a constraint can only be detected by the perspective nonlinearity
handler in a subproblem.
A typical example of this is a heuristic creating a subproblem to represent a restricted version of
the main problem.
This often involves fixing some variables or modifying bounds, which can result in new
semi-continuous variables and thus new suitable constraints.
In our test set there are 3 instances where suitable constraints were found only in subproblems.
However, subproblem detections are not guaranteed to have an impact on performance.
Because of this, subproblem detections are not counted in Table~\ref{tab:detection}.

\subsection{Overall performance impact}

This subsection evaluates the overall impact of perspective cuts on the performance of SCIP.
Its purpose is to give an overview of how the three major settings compare against each other
before moving onto more detailed comparisons in the next subsection.

In order to robustify our results against the effects of performance variability~\cite{lodi2013performance},
four different permutations of the order of variables and constraints were applied to each of the
186~instances, for which suitable structures were detected.
Each permutation is treated as a separate instance, and together with the instances without
any permutation they comprise a test set of 930 instances.
In our analysis, we exclude instances where one of the solver settings
encountered numerical troubles or where the numerical results of the different
solver settings are inconsistent.

\begin{table}[!ht]
   \small
   \caption{Overview of solved instances}
   \centering
   \label{tab:overall_instances}
   \begin{tabular}{l @{~~}|| >{\centering}p{1.6cm} | >{\centering}p{1.6cm} | >{\centering\arraybackslash}p{1.6cm} }
      \hline\noalign{\smallskip}
                              &  Off  &  Convex  &  Full \\
      \noalign{\smallskip}\hline\noalign{\smallskip}
        Solved                &  741  &  764     &  759  \\
        Limit                 &  175  &  154     &  154  \\
        Fails              &  14   &  12      &  17   \\
      \noalign{\smallskip}\hline
   \end{tabular}
\end{table}

Table~\ref{tab:overall_instances} provides an overview of the number of
instances solved to global optimality by each setting.
The row \emph{Limit} contains the count of instances where the time limit was reached.
The row \emph{Fails} reports the number of instances where numerical troubles
were encountered.
The largest number of instances solved with a given setting was 764, yielded by setting
\emph{Convex}, which had both the smallest number of numerical fails and time outs.
It is followed by \emph{Full}, which solved 759 instances.
The setting \emph{Off} solved the least number of instances overall.

\begin{table}[!ht]
   \small
   \caption{Overall results on the subset of 672 affected instances}
   \centering
   \label{tab:results_overall}
   \begin{tabular}{l @{~~}|| >{\centering}p{1.6cm} | >{\centering}p{1.6cm} | >{\centering\arraybackslash}p{1.6cm} }
      \hline\noalign{\smallskip}
               & Off   & Convex  & Full \\
      \noalign{\smallskip}\hline\noalign{\smallskip}
        Time & 13.79 & 11.23 & 11.27 \\
        Relative time & 1.00 & 0.81 & 0.82 \\
      \noalign{\smallskip}\hline\noalign{\smallskip}
        Nodes & 620 & 479 & 472 \\
        Relative nodes & 1.00 & 0.77 & 0.76 \\
      \noalign{\smallskip}\hline
   \end{tabular}
\end{table}

Table~\ref{tab:results_overall} shows the shifted geometric mean of the running time in seconds
(with a shift of 1~second) and the shifted geometric mean of the number of branch-and-bound nodes
(with a shift of 100~nodes).
All numbers in the table are computed for the subset of 672 \emph{affected} instances:
all instances where at least two of the three settings yielded a different solving path (judged by a different number of linear
programming iterations), where the solver failed with none of the settings, and solved the
instance to optimality with at least one setting.

From Table~\ref{tab:results_overall} one can see that a significant improvement is achieved when
enabling perspective cuts for convex constraints.
The results with the settings \emph{Convex} and \emph{Full}, however, are almost identical.

\subsection{Detailed comparisons}

In this subsection we provide a more detailed analysis of the performance results by comparing pairs
of settings, in order to evaluate the impact of each major feature more thoroughly.

First, we present the numbers of relevant and affected instances in Table~\ref{tab:pairwise_instances}.
It has the following rows:
\begin{itemize}
\item \emph{Relevant:} the number of instances where more expressions are detected with the second setting than
with the first setting;
\item \emph{Affected:} the number of instances which were solved with at least one of the two settings,
where the number of linear programming iterations differs between the two settings and the solver
failed with none of the two settings.
\end{itemize}

From Table~\ref{tab:pairwise_instances} we can see that while nonconvex
structures can be found on more than half of the test set, applying
generalized perspective cuts to nonconvex functions affects the solving path
less often than applying perspective cuts to convex functions.

\begin{table}[!ht]
   \small
   \caption{Relevant and affected instances}
   \centering
   \label{tab:pairwise_instances}
   \begin{tabular}{l @{~~}|| >{\centering}p{3.6cm} | >{\centering\arraybackslash}p{3.6cm} }
      \hline\noalign{\smallskip}
                              &  \emph{Off}  vs  \emph{Convex}  &  \emph{Convex}  vs  \emph{Full} \\
      \noalign{\smallskip}\hline\noalign{\smallskip}
        Relevant              &  710 & 485 \\
        Affected              &  544 & 205 \\
      \noalign{\smallskip}\hline
   \end{tabular}
\end{table}

Table~\ref{tab:rootnode} summarizes the effect of perspective cuts on the dual bound at the end of the root node.
It reports the numbers of instances of the subset \emph{Relevant} where the dual bound was better with the corresponding
setting by a percentage that is specified in the first column, as well as the numbers of instances
where the dual bound change was less than $5\%$.

\begin{table}[!ht]
   \small
   \caption{Root node dual bound differences}
   \centering
   \label{tab:rootnode}
   \begin{tabular}{l @{~~}|| >{\centering}p{1.6cm} | >{\centering}p{1.6cm} || >{\centering}p{1.6cm} | >{\centering\arraybackslash}p{1.6cm} }
      \hline\noalign{\smallskip}
                                        &  Off  &  Convex  &  Convex  &  Full \\
      \noalign{\smallskip}\hline\noalign{\smallskip}
        better by $>50\%$ &  16    &  46     &  0       &  31   \\
        better by $5$--$50\%$ & 25    &  39     &  14      &  11   \\
        same within $5\%$  &  \multicolumn{2}{c||}{584} & \multicolumn{2}{c}{429} \\
      \noalign{\smallskip}\hline
   \end{tabular}
\end{table}

A significant difference in root node dual bound can be observed only for a relatively small number
of instances.
The comparison between \emph{Off} and \emph{Convex} is consistent with the results in the above
tables, with \emph{Convex} improving more dual bounds than \emph{Off}.
The comparison between \emph{Convex} and \emph{Full} deserves a closer look.
When inspecting medium dual bound changes (5--50\%), \emph{Convex} yields a better
bound than \emph{Full} slightly more often than the other way round.
However, when considering only large improvements ($>50$\%), we observe that those were always
due to setting \emph{Full}.
From this we conclude that, overall, enabling perspective cuts for nonconvex constraints improves
the quality of dual bounds in the root node.

Table~\ref{tab:pairwise_time} compares the running time when considering pairs of settings and
the corresponding subsets of \emph{affected} instances.
It has the following rows:
\begin{itemize}
\item Time: shifted geometric mean of the running time in seconds (with a shift of 1~second);
\item Relative time: shifted geometric mean of the running time relative to the first of the two settings;
\item Faster: the number of instances where SCIP was faster with the given setting than with the other
setting by at least $25\%$.
\end{itemize}
In order to analyze the impact on subsets of increasingly hard instances, these rows repeat for
three subsets of instances given by time brackets $[t,3600]$, which contain the
instances that were solved to optimality with both settings and took at
least~$t$~seconds by at least one setting.

\begin{table}[!ht]
   \small
   \caption{Time on subsets of affected instances}
   \centering
   \label{tab:pairwise_time}
   \begin{tabular}{l @{~~}|| >{\centering}p{1.4cm} | >{\centering}p{1.4cm} || >{\centering}p{1.4cm} | >{\centering\arraybackslash}p{1.4cm} }
      \hline\noalign{\smallskip}
                              &  Off  &  Convex  &  Convex  &  Full \\
      \noalign{\smallskip}\hline\noalign{\smallskip}
        Instances in $[0,3600]$:           & \multicolumn{2}{c||}{544} & \multicolumn{2}{c}{205} \\
        \hspace{0.5cm}Time    &  12.53 & 9.70    &   24.30  &  24.82 \\
        \hspace{0.5cm}Relative time   &  1.00 &  0.77    &   1.00   &  1.02  \\
        \hspace{0.5cm}Faster  &  95   &  193     &   43     &  51    \\
      \noalign{\smallskip}\hline\noalign{\smallskip}
        Instances in $[10,3600]$: & \multicolumn{2}{c||}{276} & \multicolumn{2}{c}{149} \\
        \hspace{0.5cm}Time    & 70.96 &  45.27   &   57.47  &  59.12 \\
        \hspace{0.5cm}Relative time   & 1.00  &  0.64    &   1.00   &  1.03  \\
        \hspace{0.5cm}Faster  & 50    &  122     &   29     &  35    \\
      \noalign{\smallskip}\hline\noalign{\smallskip}
        Instances in $[100,3600]$: & \multicolumn{2}{c||}{100} & \multicolumn{2}{c}{49} \\
        \hspace{0.5cm}Time    & 506.17 & 183.90  &   263.57 &  285.85 \\
        \hspace{0.5cm}Relative time   & 1.00  &  0.36    &   1.00   &  1.08 \\
        \hspace{0.5cm}Faster  & 18    &  64      &   13     &  15   \\
      \noalign{\smallskip}\hline\noalign{\smallskip}
        Instances in $[1000,3600]$: & \multicolumn{2}{c||}{45} & \multicolumn{2}{c}{14} \\
        \hspace{0.5cm}Time    & 1444.28 & 425.60 &  814.18  & 1034.83 \\
        \hspace{0.5cm}Relative time   & 1.00   & 0.29    &  1.00    & 1.27   \\
        \hspace{0.5cm}Faster  & 10     & 32      &  5       & 5      \\
      \noalign{\smallskip}\hline
   \end{tabular}
\end{table}

The results shown in Table~\ref{tab:pairwise_time} strongly confirm that enabling perspective cuts for
convex constraints decreases the mean running time.
This effect becomes more pronounced as the difficulty of the instances increases.
On instances that took at least $100$~seconds to solve, setting \emph{Convex}
was faster almost by a factor of~3.
The additional activation of perspective cuts for nonconvex constraints, however, rather had a detrimental effect on performance, especially as
instances become more difficult.
This is despite the fact that there are more speed-ups than slow-downs when switching from setting
\emph{Convex} to setting \emph{Full}, as seen from the rows \emph{Faster}.
Hence, the increase in the mean time is due to significant slow-downs on a few challenging
instances.
However, these observed slow-downs should not be overestimated since the size of
the subsets $[100,3600]$ and $[1000,3600]$ are comparatively small.

A comparison of branch-and-bound tree sizes is given in Table~\ref{tab:pairwise_nodes}.
Again, a consistent improvement is observed when enabling perspective cuts for convex
expressions.  This improvement becomes more pronounced as the instances become more challenging.
Here we also observe an overall improvement when enabling perspective cuts for nonconvex constraints.
However, on the harder subsets $[100,3600]$ and $[1000,3600]$, \emph{Convex} still remains the best setting.

\begin{table}[!ht]
   \small
   \caption{Number of nodes on subsets of affected instances}
   \centering
   \label{tab:pairwise_nodes}
   \begin{tabular}{l @{~~}|| >{\centering}p{1.4cm} | >{\centering}p{1.4cm} || >{\centering}p{1.4cm} | >{\centering\arraybackslash}p{1.4cm} }
      \hline\noalign{\smallskip}
                              &  Off  &  Convex  &  Convex  &  Full \\
      \noalign{\smallskip}\hline\noalign{\smallskip}
        Nodes on $[0,3600]$   &  775  &  567     & 619      & 590  \\
        Relative              &  1.00 &  0.73    & 1.00     & 0.95  \\
      \noalign{\smallskip}\hline\noalign{\smallskip}
        Nodes on $[10,3600]$  &  4289 &  2436    & 1188     & 1170 \\
        Relative              & 1.00  &  0.57    & 1.00     & 0.98  \\
      \noalign{\smallskip}\hline\noalign{\smallskip}
        Nodes on $[100,3600]$ & 24924 & 7819     & 14891    & 15503 \\
        Relative              & 1.00  &  0.31    & 1.00     & 1.04 \\
      \noalign{\smallskip}\hline\noalign{\smallskip}
        Nodes on $[1000,3600]$ & 46517 & 15638   & 166889   & 199558 \\
        Relative              & 1.00   & 0.34    &  1.00    & 1.20   \\
      \noalign{\smallskip}\hline
   \end{tabular}
\end{table}

\begin{figure}[!ht]
	\begin{subfigure}{.49\linewidth}
	  \centering
	  \includegraphics[width=1\linewidth]{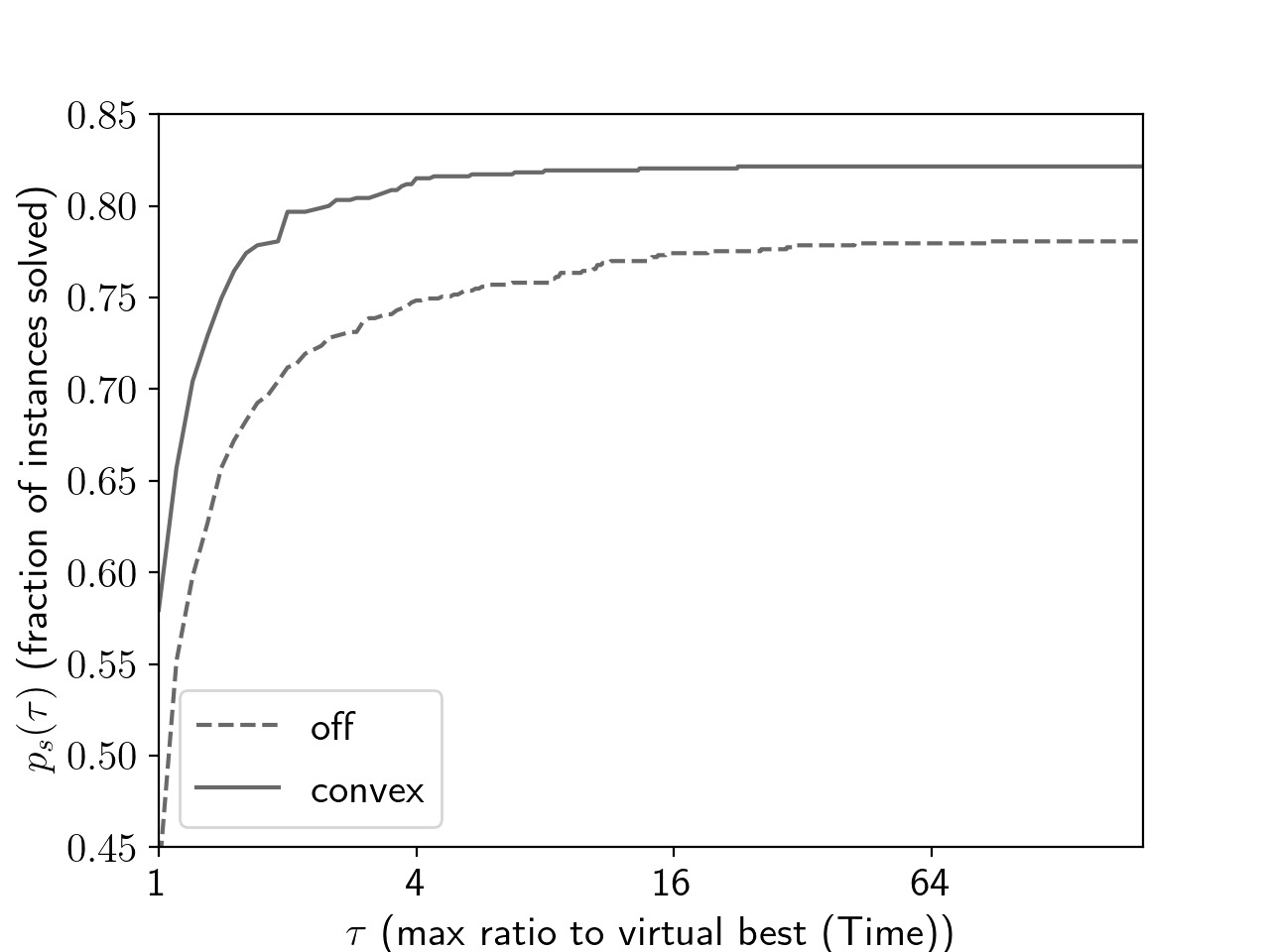}
	\end{subfigure}
	\begin{subfigure}{.49\linewidth}
	  \centering
	  \includegraphics[width=1\linewidth]{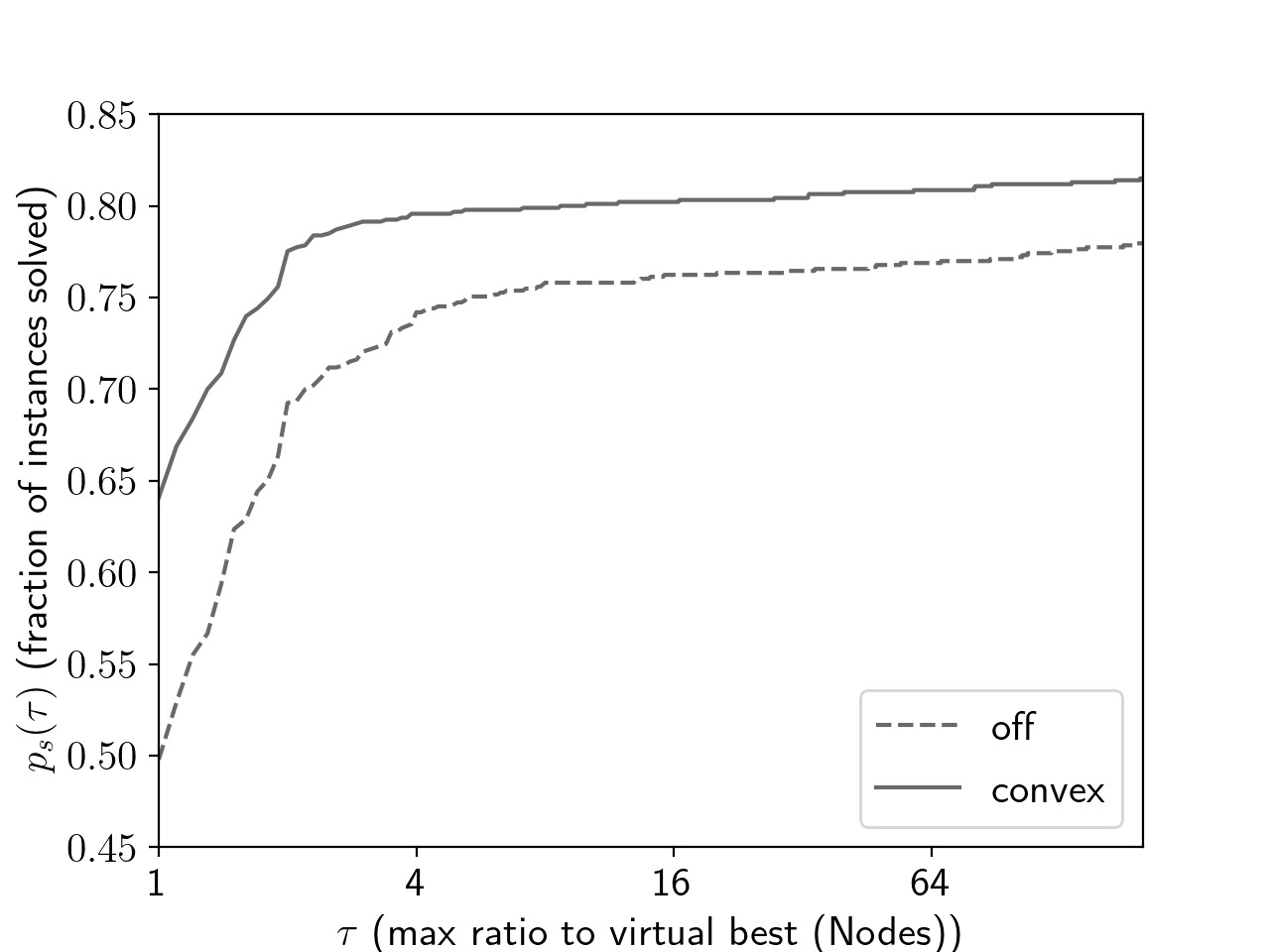}
	\end{subfigure}
	\caption{Performance profiles comparing \emph{Off} and \emph{Convex}}
	\label{fig:relppoffvsconvex}
\end{figure}

Figures~\ref{fig:relppoffvsconvex} and~\ref{fig:relppconvexvsfull} show performance
profiles~\cite{dolan2002benchmarking} for running time
and number of nodes with settings \emph{Off} and \emph{Convex} and settings \emph{Convex} and
\emph{Full}, respectively.
Let $t_{i,s}$ denote the running time for instance $i$ with setting $s$.
The virtual best setting used for the running time performance profiles, denoted by index $vb$, is
defined as a setting whose running time for each instance is equal to the minimum of the running
times with the two settings that are being compared: $t_{i,vb}=\min_s\{t_{i,s}\}$.
The horizontal axis represents the maximum allowed ratios to the time with the virtual best
setting, denoted by $\tau$.
The vertical axis represents the fraction of instances solved within the maximum allowed fraction of
time of the virtual best, denoted as $p_s(\tau)$:
$$p_s(\tau) = \frac{\text{number of instances } i \text{ s.t.: } t_{i,s} \leq \tau
\cdot t_{i,vb}}{\text{total number of instances}}.$$
Performance profiles for the number of nodes in the
branch-and-bound tree are constructed similarly,
the only difference being that $t_{.,.}$ is replaced everywhere with $n_{.,.}$, which represents
the number of nodes per instance and setting.

From Figure~\ref{fig:relppoffvsconvex} we see that setting \emph{Convex} dominates setting \emph{Off}.
With \emph{Convex}, over 55\% of
instances are solved faster or as fast as with setting \emph{Off}.
It is able to solve around $80\%$ of instances within a factor of 4 of the best time, and the curve
approaches approximately $82\%$ in the limit, i.e., \emph{Convex} is able to solve around $82\%$ of
the instances.
The respective numbers for \emph{Off} lie at around $10\%$ lower than those for \emph{Convex}.
A very similar picture is observed for the number of nodes.

\begin{figure}[!ht]
	\begin{subfigure}{.49\linewidth}
	  \centering
	  \includegraphics[width=1\linewidth]{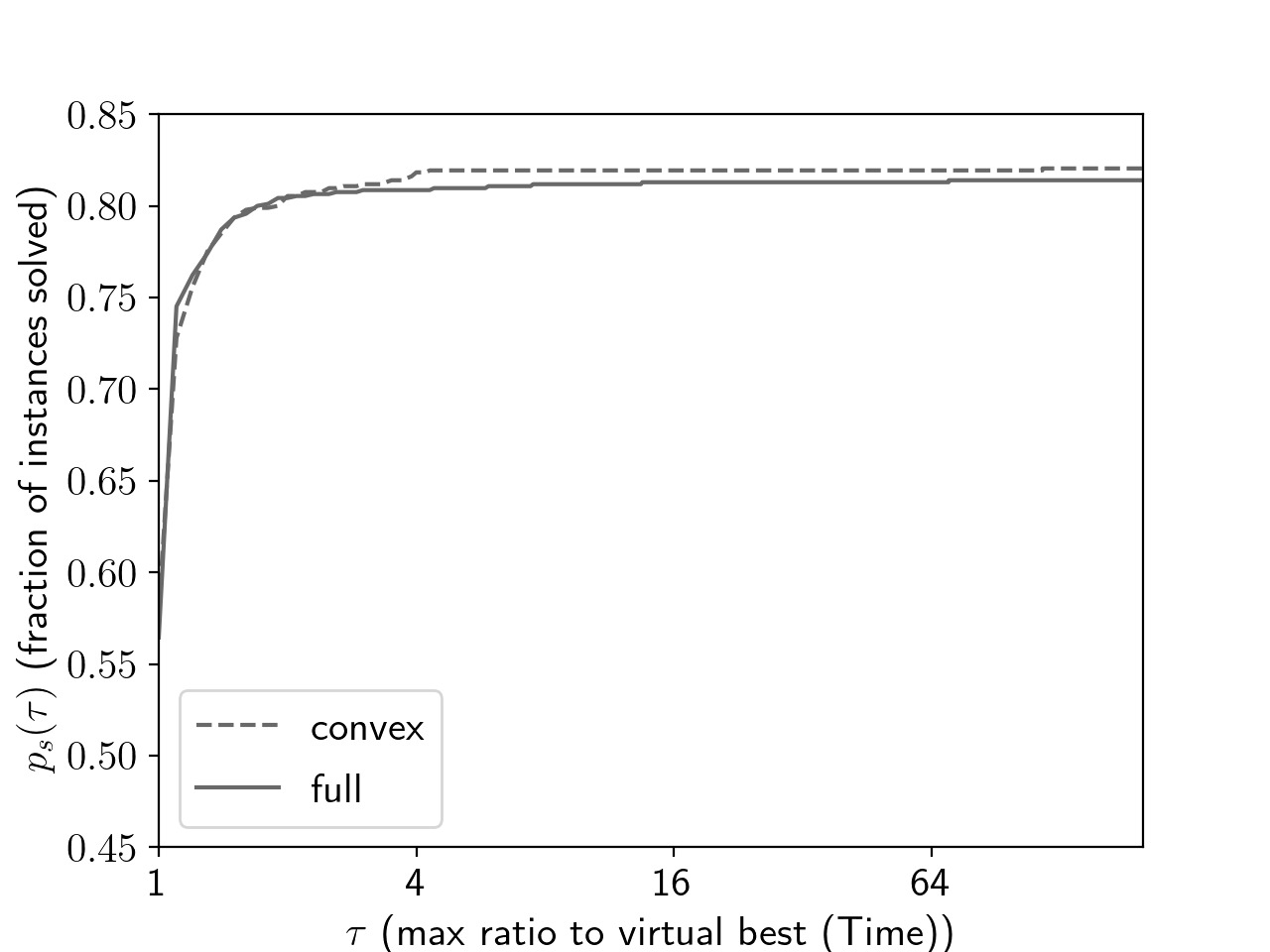}
	\end{subfigure}
	\begin{subfigure}{.49\linewidth}
	  \centering
	  \includegraphics[width=1\linewidth]{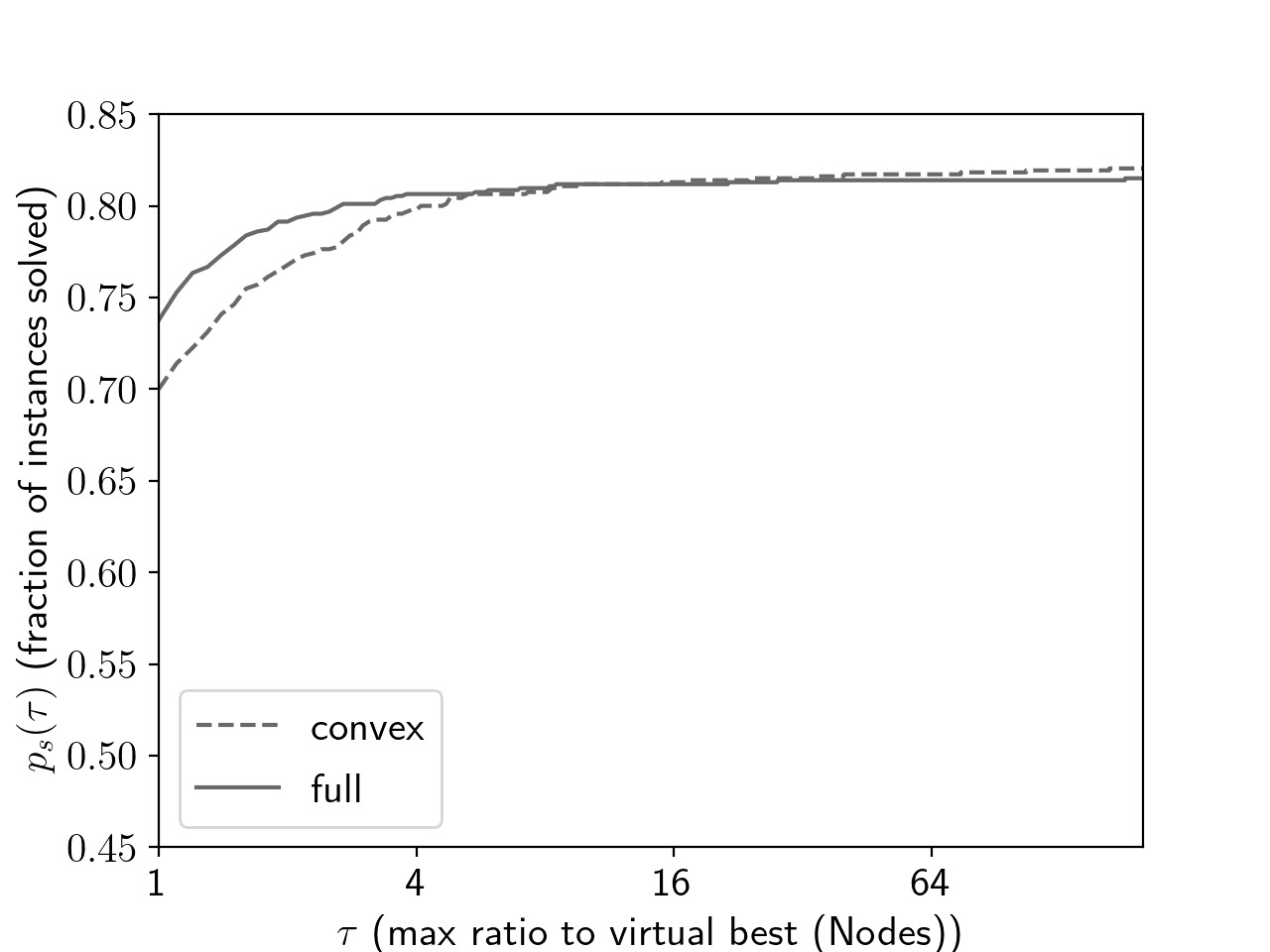}
	\end{subfigure}
	\caption{Performance profiles comparing \emph{Convex} and \emph{Full}}
	\label{fig:relppconvexvsfull}
\end{figure}

According to Figure~\ref{fig:relppconvexvsfull}, the setting \emph{Full} is roughly on par with
\emph{Convex} in terms of running time if the ratio we are interested in is below 3.
As the ratio increases, \emph{Convex} becomes the better setting, which reflects the fact that it
solves slightly more instances than \emph{Full}.
When looking at the number of nodes, \emph{Full} yields the best result for around 74\% of instances,
as opposed to around 70\% yielded by \emph{Convex}.
For ratios above 16, \emph{Convex} is better than \emph{Full}, which, again, is due to it solving
more instances.

\subsection{Feature evaluation: bound tightening}

As explained in Section~\ref{subsection:separation}, if the constraint is nonconvex and its
relaxation depends on variable bounds, then the indicator variable is first set to 1 and bound
tightening is performed.
A cut is then computed for this possibly tighter set $F^1_i$ and strengthened according to
Theorem~\ref{thm:cut_strengthening}.
In this subsection we evaluate the usefulness of this feature.

To this end, we introduce the setting \emph{Full-noBT}.
It is equivalent to \emph{Full} except that the bound tightening feature is disabled.
Table~\ref{tab:nobtvsfull} compares settings \emph{Full-noBT} and \emph{Full}
for the 68 \emph{affected} instances.

\begin{table}[!ht]
   \small
   \caption{Comparison between \emph{Full-noBT} and \emph{Full}}
   \centering
   \label{tab:nobtvsfull}
   \begin{tabular}{l @{~~}| >{\centering}p{1.2cm} | >{\centering}p{1.2cm} | >{\centering}p{1.2cm} | >{\centering}p{1.27cm} | >{\centering}p{1.2cm} | >{\centering\arraybackslash}p{1.2cm} }
      \hline\noalign{\smallskip}
                  &  Fails &   Limit  &  Solved  &  RootImpr $>50\%$  &  Time   &  Nodes \\
      \noalign{\smallskip}\hline\noalign{\smallskip}
        Full-noBT &  16    &   153    & 761      &  4                 &  34.45  &  2910  \\
        Full      &  17    &   154    & 759      &  25                &  33.68  &  2618  \\
      \noalign{\smallskip}\hline
   \end{tabular}
\end{table}

When bound tightening is disabled, two more instances are solved due to one less fail and one less
time out.
Enabling it, on the other hand, leads to large ($>50\%$) root node dual bounds improvements on
25 out of 68 affected instances and a comparable weakening of root node dual bounds only on 4
instances.
Enabling bound tightening also yields a small
decrease in the mean time ($2.2\%$) and a moderate decrease in the number of nodes ($10\%$).
According to these results, the two settings are very close in performance, \emph{Full-noBT} being
the slightly more reliable setting and \emph{Full} yielding smaller branch-and-bound trees and
slightly better solving times.

\begin{figure}[!ht]
	\begin{subfigure}{.49\linewidth}
	  \centering
	  \includegraphics[width=1\linewidth]{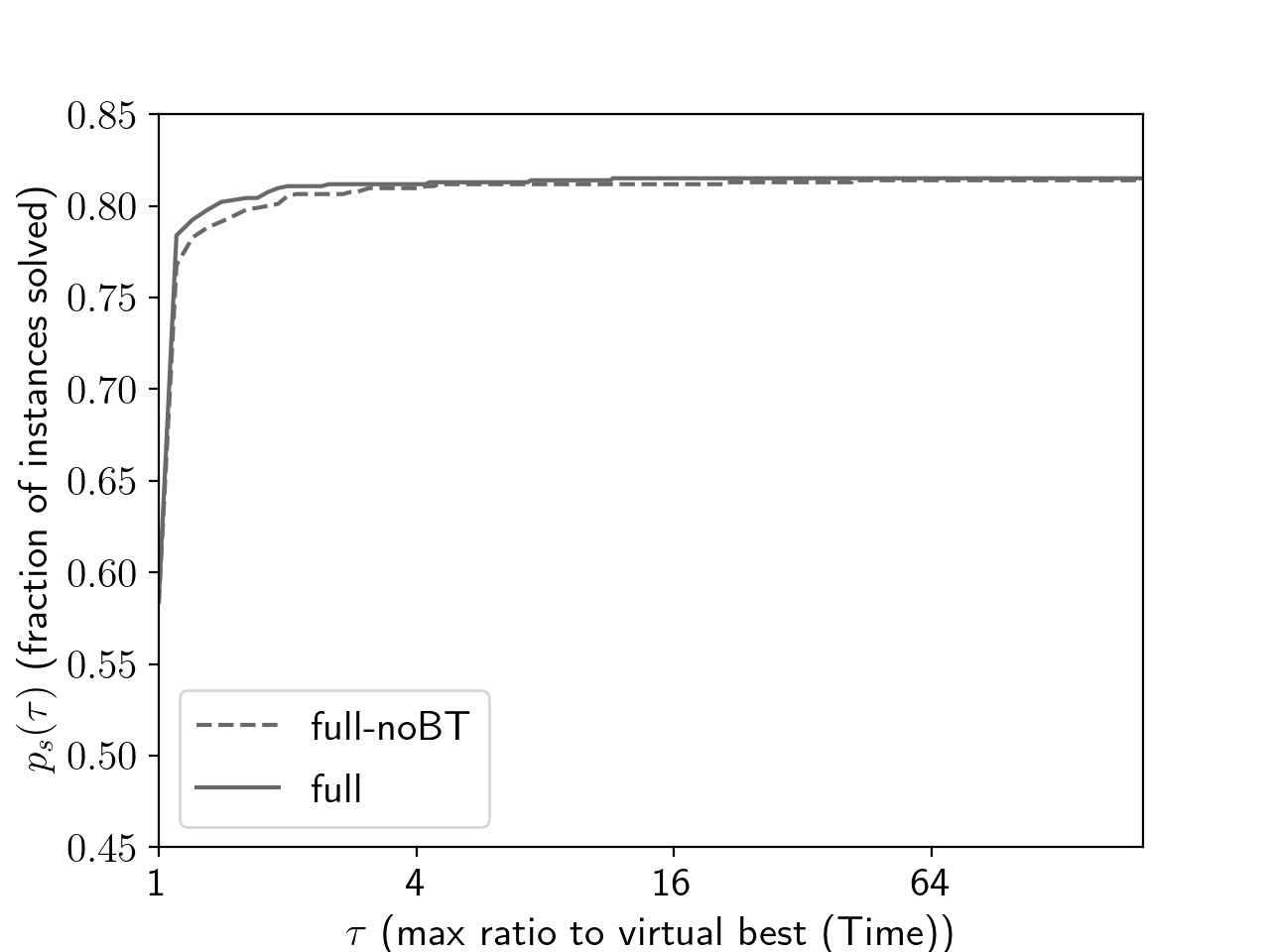}
	\end{subfigure}
	\begin{subfigure}{.49\linewidth}
	  \centering
	  \includegraphics[width=1\linewidth]{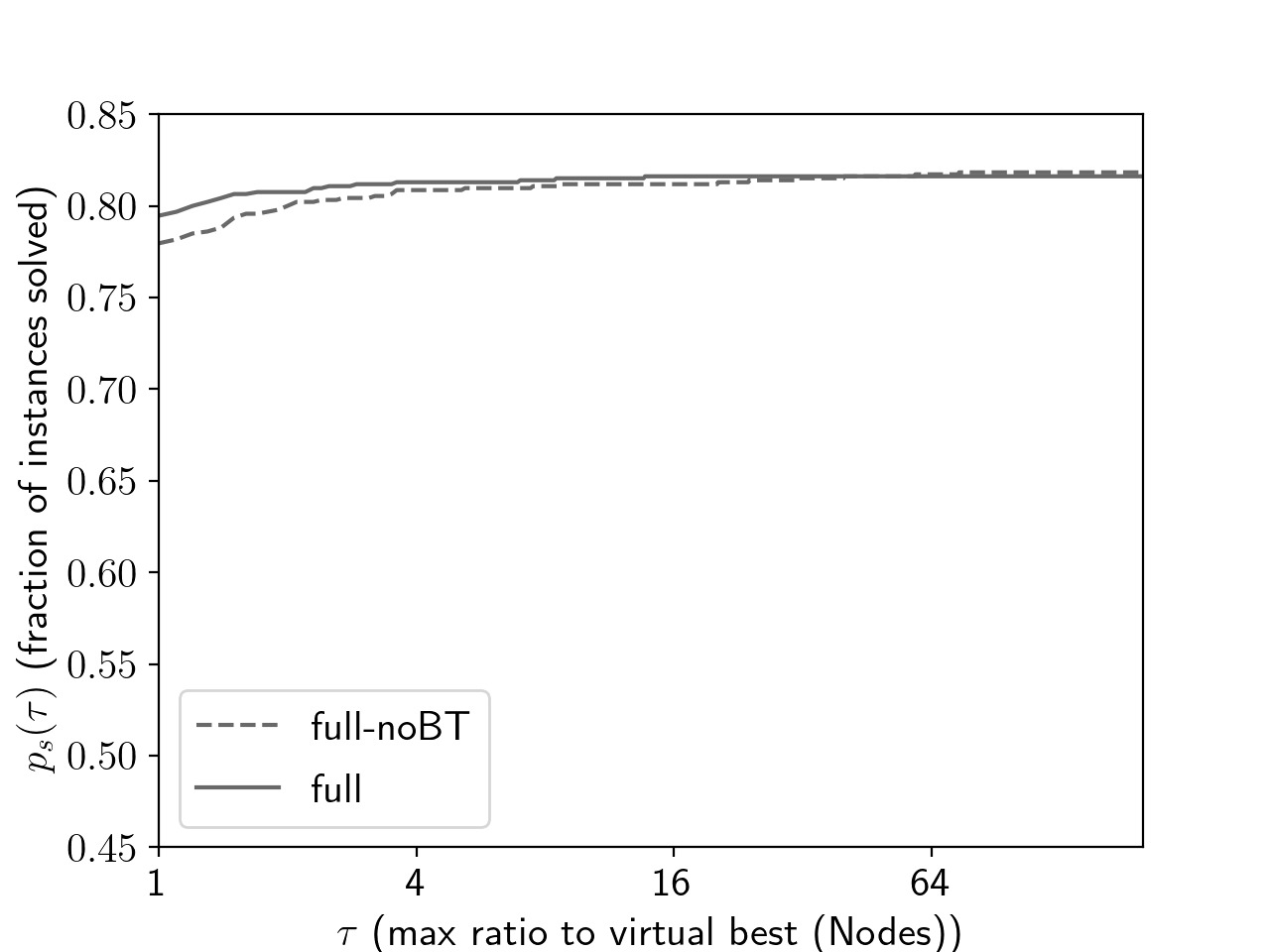}
	\end{subfigure}
	\caption{Performance profiles comparing \emph{Full-noBT} and \emph{Full}}
	\label{fig:relppnoBTvsFull}
\end{figure}

Performance profiles comparing \emph{Full-noBT} and \emph{Full}
are shown in Figure~\ref{fig:relppnoBTvsFull}.
The curves are very close since there are few affected instances.
The setting \emph{Full} performed slightly better than \emph{Full-noBT} both in terms of running
times and tree sizes, and the two settings are nearly identical in the limit.

\section{Conclusion}

In this paper we introduced a general method to construct perspective cuts not only for convex
constraints as previously proposed in the literature, but also for nonconvex
constraints, for which linear underestimators are readily available.
We conducted a computational study of perspective cuts for convex and nonconvex
constraints.
Relevant structures were detected in about 10\% of MINLPLib instances.
The computational results indicate that adding perspective cuts for convex constraints
reduces the mean running times and tree sizes by over 20\%.
Adding perspective cuts for nonconvex constraints can be detrimental to performance on challenging
instances and can lead to an increased amount of numerical issues, which is reflected in a small decrease in the number of solved instances.
Despite this, perspective cuts for nonconvex constraints reduce the geometric mean of the number of
nodes of the branch-and-bound tree by 5\% and improve dual bounds at the root node.

These results indicate that perspective cuts improve performance of ge\-ne\-ral-purpose solvers.
However, in order to efficiently utilize perspective cuts for nonconvex structures, careful
implementation and tuning is necessary.

One direction for future work is developing more sophisticated detection algorithms.
Some problems contain constraints that do not satisfy the requirements in our current
implementation, but with more careful analysis of the problem structure can be revealed
to be suitable for applying perspective cuts.
Another direction for future research is
generalizing the cut strengthening method to on/off variables whose ``off''
domain is a non-singleton interval, as well as to more general types of on/off sets.
Unlike the case considered in this paper, there is no single best choice of how to strengthen a
given valid cut in such a setting.
Moreover, increased complexity of the set makes developing a computationally efficient cut
strengthening method a more challenging task.

\section*{Funding}

The work for this article has been conducted within the Research Campus
MODAL funded by the German Federal Ministry of Education and Research (BMBF grant numbers 05M14ZAM,
05M20ZBM).

The described research activities are funded by the Federal Ministry for Economic Affairs and
Energy within the project EnBA-M (ID: 03ET1549D).

\bibliographystyle{spmpsci}
\bibliography{persp}

\end{document}